\documentclass[10pt]{elsarticle}
\usepackage[cp1251]{inputenc}
\usepackage[english]{babel}
\usepackage{amsmath}
\usepackage{amssymb}
\usepackage{amsfonts}
\usepackage{rotating}
\usepackage{graphicx}
\usepackage{floatflt,epsfig}
\usepackage{lineno,hyperref}
\usepackage{enumerate}
\usepackage{colortbl}
\usepackage[table]{xcolor}
\usepackage{array,tabularx,tabulary,booktabs}
\usepackage{longtable}
\usepackage{multirow}
\usepackage{wrapfig}
\usepackage{subcaption}

\newcolumntype{^}{>{\currentrowstyle}}

\journal{Journal of Combinatorial Designs}
\newtheorem{thm}{Theorem}
\newtheorem{lem}{Lemma}
\newtheorem{prop}{Proposition}

\newtheorem{rem}{Remark}

\begin{document}
\renewcommand{\abstractname}{Abstract}
\renewcommand{\refname}{References}
\renewcommand{\tablename}{Figure.}
\renewcommand{\arraystretch}{0.9}
\thispagestyle{empty}
\sloppy

\begin{frontmatter}
\title{Deza graphs with parameters $(v,k,k-2,a)$\tnoteref{grant}}
\tnotetext[grant]{
Both authors are partially supported by RFBR according to the research project 17-51-560008.
}

\author[01]{Vladislav~V.~Kabanov}
\ead{vvk@imm.uran.ru}
\author[01,02]{Leonid~Shalaginov}
\ead{44sh@mail.ru}

\address[01]{Krasovskii Institute of Mathematics and Mechanics, S. Kovalevskaja st. 16, Yekaterinburg, 620990, Russia}
\address[02]{Chelyabinsk State University, Brat'ev Kashirinyh st. 129, Chelyabinsk,  454021, Russia}

\begin{abstract}
A Deza graph with parameters $(v,k,b,a)$ is a $k$-regular graph on $v$ vertices in which the number of common neighbors  of two distinct vertices takes two values $a$ or $b$ ($a\leq b$) and both  cases exist.
In the previous papers \cite{KMS, GHKS} Deza graphs with parameters $(v,k,b,a)$ where $k-b = 1$ were characterized.
In this paper we characterise Deza graphs with  $k-b = 2$. 
\end{abstract}

\begin{keyword}
Deza graph; divisible design graph; strongly regular graph.
\vspace{\baselineskip}
\MSC[2010] 	05C75\sep 05B30\sep 05E30
\end{keyword}
\end{frontmatter}

\section{Introduction}
The graphs studied in this paper are finite undirected graphs without loops and multiple edges.
{\em A Deza graph} $G$ with parameters $(v,k,b,a)$ is a $k$-regular graph on $v$ vertices  in which the number of common neighbors  of two distinct vertices takes two values $a$ or $b$ ($a\leq b$) and both  cases exist.

The concept of a Deza graph was introduced  by M. Erickson, S. Fernando, W. Haemers, D. Hardy, and J. Hemmeter in   \cite{EFHHH}. It was influenced by the paper of A. Deza and M. Deza \cite{dd}.   

 In comparison with a strongly regular graph, a Deza graph (in case $a=0$) can have diameter more than 2. If a Deza graph has diameter 2 and is not strongly regular, then it is called {\em a strictly Deza graph}. If $G$ is a strongly regular graph, then the quadruple of parameters $(v, k, \lambda, \mu)$ is used, where  $\lambda$ is equal to the number of common neighbors of every two adjacent vertices of $G$ and $\mu$ is equal to the number of common neighbors of every two distinct non-adjacent vertices of $G$. Thus, the notion of Deza graphs is a generalization of the notion of strongly regular graphs in such a way that the number of common neighbors of any pair of distinct vertices in a Deza graph does not depend on the adjacency. 

The authors of  \cite{EFHHH} developed a basic theory of strictly Deza graphs and introduced a few constructions of such graphs. They also found all strictly Deza graphs
with the number of vertices at most 13.  S. Goryainov and L. Shalaginov in \cite{gsh1} found all strictly Deza graphs which have the number of vertices equals to 14, 15, or 16.   Deza graphs can have applications in several fields of discrete mathematics especially in design theory finite geometries,  and connected point-block incidence structures. 

A connected graph $G$ is called a $(0,\lambda)$-graph if any two distinct vertices in $G$ have exactly $\lambda$ common neighbors or none at all.  $(0,\lambda)$-graphs  were introduced and studied by M. Malder in \cite{MM}.
He proved that in case $\lambda\geq 2$ such graphs are regular. Therefore,  $(0,\lambda)$-graphs with $\lambda\geq 2$  are Deza graphs. M. Malder proved that if $G$ is $k$-regular graph on $v$ vertices and has the diameter $d$, then $v\leq 2^k$ and $d\leq k$. In both cases equality  is true only  for the $n$-dimensional binary cube (the hypercube) when $\lambda = 2$.  
All $(0, 2)$-graphs  of valency at most 8 was found by A.~E.~Brouwer in \cite{AB}  and
A.~E.~Brouwer, P.~R.~J.~\"{O}sterg\r{a}rd in \cite{BO}. In general, the complement of a Deza graph isn't a Deza graph. Let's note that the complement of any $(0, 2)$-graph is a strictly Deza graph.

In \cite{HKM} W.H.~Haemers, H.~Kharaghani, and M.~Meulenberg introduced and studied a notion of divisible design graphs. A $k$-regular graph on $v$ vertices is a {\em divisible design graph} (DDG for short) with parameters $(v,k,\lambda_1 ,\lambda_2 ,m,n)$ if the vertex set can be partitioned into $m$ classes of size $n$, such that two distinct vertices from the same class have exactly $\lambda_1$ common neighbors, and two vertices from different classes have exactly $\lambda_2$ common neighbors.
Divisible design graph with $m = 1$, $n = 1$, or $\lambda_1 = \lambda_2 $ is called  improper, otherwise it is called proper.
Divisible design graphs are a special case of the notion of Deza graphs. Moreover, Deza graphs with $a<2b-k$ are divisible design graphs (see Proposition~\ref{a<2b-k}). 
 
Strongly regular graphs with parameters $(v, k, \lambda, \mu)$ such that $k = \mu$ are known  (Theorem 1.3.1(v) in \cite{BCN}).
 Deza graphs with parameters $(v, k, b, a)$  such that $b = k$ was obtained  in the  paper  \cite[Theorem 2.6]{EFHHH}.
Deza graphs with $b = k-1$ were characterised in \cite{KMS} and \cite{GHKS}. 

In this paper Deza graphs with parameters $(v, k, b, a)$ and $b=k-2$  are studied.

Let $G$ be a graph. If  $x\in V(G)$ then  the set of all neighbors of $x$ in $G$ we denote by $N(x)$.  
The set of all vertices at distance precisely $2$ from $x$ in $G$ we denote by $N_2(x)$. If $B$ is a set of vertices of $G$, then $N(B)$ is used  to denote the union of the neighborhoods of the vertices of $B$. 

The paper organized as follows. At first we consider Deza graphs with parameters $(v,k,k-2,a)$, where $a=0$. We consider only
 connected Deza graphs. If a Deza graph is disconnected, then 
 $a=0$ and Theorem \ref{Th a=0} describes its connected components.

\begin{thm}\label{Th a=0}
Let $G$ be a connected Deza graph with parameters $(v,k,k-2,0)$. Then one of the following cases holds:
\begin{enumerate}
    \item[$1.$] $G$ is the strictly Deza graph with parameters $(8,4,2,0)$ and it is isomorphic to the $4\times n$-grid.
    \item[$2.$] $G$ is the Deza graph with parameters $(14,4,2,0)$ and it is isomorphic to the non-incidence graph of the Fano plane;
    \item[$3.$] $G$ is isomorphic to the four-dimensional binary cube $H(4,2)$,
    \item[$4.$] $G$ has parameters $(v,3,1,0)$ and diameter more than $2$. In particular, if $G$ has parameters $(14,3,1,0)$, then it is isomorphic to the incidence graph of the Fano plane;
    \item[$5.$] $G$ is the Petersen graph.
\end{enumerate} 
\end{thm}

Our next step is to prove that if $G$ is a Deza graph with parameters $(v,k,b,a)$ and $a<2b-k$, then $G$ is a DDG. When $a\geq 2b-k$ we have $a\in \{k-3, k-4\}$. 

\begin{thm}
Let $G$ be a connected Deza graph with parameters $(v,k,k-2,a)$.
Then either $a\in \{k-3, k-4\}$ or $G$ is a DDG. 

\item[$1.$] If $G$ is a Deza graph with parameters $(v,k,k-2,k-3)$, then one of the following cases holds:
\begin{enumerate}
    \item[$(i)$] $G$ is one of the strictly Deza graphs with parameters $(8,4,2,1)$ or $(9,4,2,1)$;
    \item[$(ii)$] $G$ is one of the strongly regular graphs with parameters $(9,4,1,2)$, $(10,3,0,1)$ or $(10,6,3,4)$;
    \item[$(iii)$] $G$ has parameters $(v,3,1,0)$ and diameter more than $2$.
\end{enumerate}
\item[$2.$] If $G$ is a Deza graph with parameters $(v,k,k-2,k-4)$, then one of the following cases holds:
\begin{enumerate}
    \item[$(i)$]  $G$ is isomorphic to the complement of the  disjoint union of $s$ cubes $H(3,2)$ and it has parameters $(8s,8(s-1)+4, 8(s-1)+2, 8(s-1))$;
    \item[$(ii)$] $G$ is the Deza graph with parameters $(14,4,2,0)$ and it is isomorphic to the non-incidence graph of the Fano plane;
    \item[$(iii)$] $G$ is isomorphic to the four-dimensional cube with parameters $(16,4,2,0)$.
\end{enumerate}
\end{thm}

In the last section we consider divisible design graphs with parameters $(v, k, \lambda_1 ,\lambda_2, m, n)$, where $k-2\in \{ \lambda_1, \lambda_2 \}$. 
In particular, we prove that $0\in \{\lambda_1, \lambda_2\}$. 

\begin{thm}
Let $G$ be a DDG with parameters $(v,k,\lambda_1, \lambda_2,m,n)$ and $k-2\in \{ \lambda_1, \lambda_2 \}$. 
Then one of the following statements holds:
\item[$1.$] $G$ has parameters $(14,4,2,0,2,7)$ and it is isomorphic to the
non-incidence graph of points and lines of the Fano plane;
\item[$2.$] $G$ has parameters $(14,3,1,0,2,7)$ and it is isomorphic to the
incidence graph of points and lines of the Fano plane;
\item[$3.$] $G$ has parameters $(8,4,2,0,2,4)$ and it is isomorphic to the $4\times n$-grid.
\end{thm}

\section{Case of a=0, a=k-3 or a=k-4}

In this section we consider some general properties of Deza graphs.  Also we consider Deza graphs with restriction on parameters $b=k-2$ and $a = 0$, $a = k-3$ or $a = k-4$. Moreover, we get a sufficient condition for a Deza graph to be a divisible design graph.

Let $G$ be a strictly Deza graph with parameters $(v,k,b,a)$. Let also $\alpha (x)$ be the number of  vertices $y\in V(G)$ such that $|N(x)\cap N(y)| = a$ and $\beta (x)$ be the number of  vertices $y\in V(G)$ such that $|N(x)\cap N(y)| = b$.
It is clear that $v = 1 + \alpha (x) + \beta (x)$ for any vertex $x\in V(G)$. 

\begin{prop}\label{beta}
\item[$1.$] The following equality holds for any vertex $x\in V(G)$:
$$\beta := \beta (x) = \cfrac{k(k - 1) - a(n - 1)}{b - a}.$$
\item[$2.$]  $k(k - 1) - a(n - 1)$ is divided by $b-a$.
\end{prop}
\begin{proof} It follows from Proposition 1.1 in \cite{EFHHH}.
\end{proof}

\begin{prop}\label{a<2b-k} Let $G$ be a Deza graph with parameters $(v,k,b,a)$ where $a<2b-k$ then $G$ is a DDG. \end{prop}

\begin{proof} Let's consider a binary relation $\rho$ on $V(G)$. Let $\rho$ be "coincide or have b common neighbors". If $\rho$ is equivalence then equivalent classes are classes of the canonical partition of DDG. It is clear that $\rho$ is  an equivalence relation if $\beta  = 1$.
Suppose that $\beta \geq 2$ and there are vertices $x$, $y$, $z$ such that  $(x,y) \in \rho$ and $(y,z) \in \rho$. Then $x$ and $z$ have at least $2b-k$ common neighbors in $N(y)$. If $a<2b-k$ then $x$ and $z$ have $b$ common neighbors. So $(x,z) \in \rho$.
$\square$\end{proof}

\begin{prop}\label{a=0}
Let $G$ be a  Deza graph with parameters $(v,k,k-2,0)$. Then any connected component of $G$ is isomorphic one of the following graphs:
\begin{enumerate}
    \item[$1.$] $G$ is the strictly Deza graph with parameters $(8,4,2,0)$,
    \item[$2.$] $G$ is isomorphic to the four-dimensional binary cube $H(4,2)$,
    \item[$3.$] $G$ is the Deza graph with parameters $(14,4,2,0)$ which is  isomorphic to the non-incidence graph of the Fano plane,
    \item[$4.$] $G$ has parameters $(v,3,1,0)$ and diameter more than $2$,
    \item[$5.$] $G$ is the Petersen graph.
\end{enumerate} 
\end{prop}

\begin{proof} Let $G$ be a connected Deza graph with parameters $(v,k,k-2,0)$. Consider two cases: graph $G$ has a triangle and $G$ has no triangles.

Suppose $G$ has a triangle and this triangle is induced by $x$, $y$, $z$. Then $x$ and $y$ have the common neighbor $z$. Since $a=0$, we have $|N(x)\cap N(y)| = k-2$ and $y$ adjacent with all vertices in $N(x)$ exclude one vertex, say $w$. Now we have two cases. If $x$ and $w$ don't belong to a triangle, then $N(x)\setminus \{w\}$ induces a complete graph. If $x$ and $w$ belong to a triangle, then $N(x)$ induces a complete graph with removed perfect matching.

At first, we consider the case when $N(x)\setminus \{w\}$ induces a complete graph. 
Let's calculate the number of edges between $N(x)$ and $N_2(x)$. We have $k-1$ 
edges between $N(x)\setminus \{w\}$ and $N_2(x)$. Also $k-1$ edges between $w$ and $N_2(x)$. Thus, there are $2(k-1)$ edges between $N(x)$ and $N_2(x)$. On the other hand, there are $|N_2(x)|(k-2)$ edges  between  $N_2(x)$ and $N(x)$. Let $|N_2(x)|=t$. Then $2(k-1) = t(k-2)$ and it implies that either $k = 3$, $t = 4$ or $k = 4$, $t = 3$. It is easy to see that if $k=4$ and $t=3$ then $G$ is the strictly Deza graph with parameters $(8,4,2,0)$ and if $k=3$, $t=4$ then $G$ has parameters $(v,3,1,0)$ and diameter more than 2.

The second case, when $N(x)$ induces a complete graph with removed perfect matching, is impossible, because vertices $y$ and $w$ have $k-2$ common neighbors in $N(x)$ and $x$ is their common neighbor too. It is a contradiction.

Now suppose graph $G$ has no triangles. If we calculate edges between $N(x)$ and $N_2(x)$ in two ways, then we have equation $k(k-1) = t(k-2)$, where $t = |N_2(x)|$. Since $k-1$ and $k-2$ are mutually prime integers, then  $k-2$ divides $k$ and this implies that $k=3$ or $k = 4$. If $k=3$ then $G$ has parameters $(v,3,1,0)$.  Thus, either $ G $ is a Petersen graph, or its diameter is greater than 2. If $k=4$ then $G$ is the four-dimensional binary cube or $G$ is a graph with parameters $(14,4,2,0)$ that is isomorphic to the non-incidence graph of the Fano plane (for more details see \cite{MM} and \cite[Tables 1, 2]{AB}).
$\square$\end{proof}

 Theorem 1 follows from Proposition \ref{a=0}.  
\begin{prop}\label{a=k-3}
Let $G$ be a Deza graph with parameters $(v,k,k-2,k-3)$ and $k-3\neq 0$.
Then one of the following cases holds:
\begin{enumerate}
    \item[$1.$] $G$ is one of the strictly Deza graphs with parameters $(8,4,2,1)$ or $(9,4,2,1)$,
    \item[$2.$] $G$ is one of the strongly regular graphs with parameters $(9,4,1,2)$,  or $(10,6,3,4)$,
\end{enumerate}
\end{prop}
\begin{proof}
By Proposition \ref{beta} we have $\beta = \cfrac{a(v-1) - k(k-1)}{a-b} >0$.

Since $k-3\neq 0$, then $v<k+3+\cfrac{6}{k-3}$. On the other hand, if we consider non-adjacent vertices we have either $v\geq 2k - b +2 = k+4$ or $v\geq 2k - a +2 = k+5$.   
In any case,  $k+3+\cfrac{6}{k-3} > k+4$, then $k<9$ and $v<17$.
But all such Deza graphs are known \cite{EFHHH, gsh1}. In this case $G$ is one of the strictly Deza graphs with parameters $(8,4,2,1)$ or $(9,4,2,1)$. If $G$ is a strongly regular graph, then $G$ has parameters $(9,4,1,2)$ or  $(10,6,3,4)$. 
$\square$\end{proof}

\begin{prop}\label{a=k-4}
Let $G$ be a Deza graph with parameters $(v,k,k-2,k-4)$ and $k-4\neq 0$. Then 
 $G$ is isomorphic to the complement of the  disjoint union of $s$ cubes $H(3,2)$ for some integer $s$ and its parameters are $(8s,8(s-1)+4, 8(s-1)+2, 8(s-1))$.
\end{prop}

\begin{proof}
By Proposition \ref{beta} we have $\beta = \cfrac{a(v-1) - k(k-1)}{a-b}>0$. 

Since  $k-4\neq 0$, then $v<k+4+\cfrac{12}{k-4}$. On the other hand, if we consider non-adjacent vertices we have either $v\geq 2k - b +2 = k+4$ or $v\geq 2k - a +2 = k+6$. In any case, $v\geq  k+4$. If $v = k+4$ then $G$ has parameters $(k+4,k,k-2,k-4)$. The complement of $G$ has parameters $(k+4,3,2,0)$ hence $\overline{G}$ is isomorphic to  the  disjoint union of some (say $s$) cubes $ H(3,2)$. So the parameters of $G$ are $(8s,8(s-1)+4, 8(s-1)+2, 8(s-1))$.

If $v > k+4$ then $k<16$ and $v < 19$. In the case $v < 16$ by  \cite{EFHHH,gsh1} we don't have any graphs. In the case
$16 \geq v < 19$  we have the parameters sets $(17,4,2,0)$, $(18,4,2,0)$, $(18,5,3,1)$, and  $(18,13,11,9)$. By \cite[Theorem 12]{MM}, there are no Deza graphs with parameters $(17,4,2,0)$, $(18,4,2,0)$. In the case $(18,5,3,1)$ and  $(18,13,11,9)$: $v$ is even and $k,b,a$ are odd. It is impossible by Proposition \ref{beta}.
$\square$\end{proof}

Now Theorem 2 follows from Propositions \ref{a=0}, \ref{a=k-3}, and \ref{a=k-4}.

\section{Divisible design graphs}

Let $G$ be a divisible design graph with parameters $(v,k,\lambda_1,\lambda_2,m,n)$. Since $G$ is a Deza graph, then it has  parameters $(v,k,b,a)$, where $\{\lambda_1,\lambda_2\}=\{b,a\}$. 
By Propositions \ref{a=k-3}, \ref{a=k-4}, there are no DDGs with parameters
$\{\lambda_1,\lambda_2\}=\{k-2,k-3\}$ and $\{\lambda_1,\lambda_2\}=\{k-2,k-4\}$ when $a\neq 0$. Further we consider the case $a\neq 0$.

\begin{prop}\label{a not 0} There are no  divisible design graphs with parameters $(v,k,\lambda_1,\lambda_2,m,n)$, where $\{\lambda_1,\lambda_2\} = \{k-2,a\}$ and $0< a < k-4$.
\end{prop}

\begin{proof} Further we prove this proposition in a number of lemmas. Let $G$ be a divisible design graphs with parameters $(v,k,\lambda_1,\lambda_2,m,n)$, where $\{\lambda_1,\lambda_2\} = \{k-2,a\}$ and $0< a < k-4$.

\begin{lem}\label{Basic}
Let $G$ be a DDG with parameters $(v,k,\lambda_1,\lambda_2,m,n)$ and $k-2 \in \{\lambda_1,\lambda_2\}$. Then the following properties hold: 
\begin{enumerate}
    \item[$(i)$] $\lambda_1 = k-2$,
    \item[$(ii)$] $n$ divides $k^2 - 2$,
    \item[$(iii)$] $n\neq 3,4,5,6$.
\end{enumerate} 
\end{lem}
\begin{proof}
$(i)$ DDG with parameters $(v,k,\lambda_1,\lambda_2,m,n)$ has real eigenvalues  $$\{k, \pm\sqrt{k-\lambda_1},\pm\sqrt{k^2-\lambda_2 v} \}$$ \cite[Lemma 2.1]{HKM}. If  $\lambda_2 = k-2$  then $k^2 > (k-2)v$  and hence  $v < k+2+\cfrac{4}{k-2}.$ It is a contradiction. Therefore, $\lambda_1 = k-2$.

$(ii)$ From \cite[equation (1)]{HKM} follows that $k^2 = k + \lambda_1 (n-1)+\lambda_2 n(m-1) = k+(k-2)(n-1)+an(m-1) = (k-2+a(m-1))n + 2$. Hence, $n$ divides $k^2-2$. 

$(iii)$ Since $k^2\equiv 2 (mod\ n)$, then $2$ is a quadratic residue modulo $n$ and $n \neq 3,4,5,6$.
$\square$\end{proof}
\medskip

A partition $\pi = \{B_1, B_2, \ldots , B_m\}$ of the vertices of a graph $G$ is
{\em equitable} if for every pair of indices $i, j\in \{1,\ldots , m\}$, which are not necessarily distinct, there is a non-negative integer $b_{i,j}$ such that each vertex $x$ in $B_i$ has exactly $b_{i,j}$ neighbors in $B_j$, regardless of the choice of $x$.

The vertex partition from the definition of a DDG is called the {\em  canonical partition}.

\begin{lem}\label{equitable}
The canonical partition of  a proper  DDG is equitable. 
\end{lem}
\begin{proof}
This is Theorem 3.1 \cite{HKM}.
$\square$\end{proof}

\begin{lem}
Let $G$ be a DDG with parameters $(v,k,k-2,a,m,n)$ then each class of the canonical partition of $G$ is a coclique.
\end{lem}
\begin{proof}
Let $B$ be a class of the canonical partition of $G$.
Let us consider $x,y\in B$ such that $x$ is adjacent to $y$. Since $\lambda_1 = k-2$, then 
$|N(x)\cap N(y)| = k-2$ and there is the only vertex $z$ in $N(x)\setminus N(y)$. Each common neighbor of $x$ and $z$ also is a neighbor of $y$. Then
$$|N(y)\cap N(z)| \geq |(N(x)\cap N(z))\cup  \{x\}| > |(N(x)\cap N(z))|.$$  
Hence we have a contradiction to the fact that $B$ is a class of the canonical partition of $G$. 
$\square$\end{proof}
\medskip

If $B$ is a set from $V(G)$, then we denote $\bigcap\limits_{x\in B}N(x)$ by $W(B)$. Denote by $B_w$ the class of canonical partition of $G$ containing a vertex $w$.

\begin{lem}\label{n_divides_W}
Let $G$ be a DDG with parameters $(v,k,\lambda_1,\lambda_2,m,n)$ and let $B$ be a class of canonical partition of $G$. If $w\in W(B)$, then $B_w\subseteq W(B)$ and $n$ divides $|W(B)|$.
\end{lem}
\begin{proof}
Let $w\in W(B)$.  Then $w$ adjacent with all vertices from $B$ but since the canonical partition of $G$ is equitable each vertex from $B_w$ adjacent with all vertices from $B$. Hence, for each vertex $w\in W(B)$ we have $B_w\subseteq W(B)$ and $n$ divides $|W(B)|$. 
$\square$\end{proof}
\medskip

Since $n \neq 3,4,5,6$ then we need to study two cases: $n=2$ and $n\geq 7$.

Let's first begin with the case $n = 2$. 

\begin{prop}\label{n=2}
There are no DDGs with parameters $(2m,k,k-2,a,m,n)$, where $n=2,\ a\neq 0$.
\end{prop}
\begin{proof}
Let $G$ be a DDG with parameters $(2m,k,k-2,a,m,2)$. Then  \cite[equation 1]{HKM} implies that  $$k^2 = k + (k-2)(2-1)+2a(m-1) = 2k - 2a - 2 + 2am.$$ Hence $k$ is an  even integer and $k^2 = 2(k+a(m-1)-1)$. Thus, $a(m-1)$ is an odd integer and hence $m$ is an even integer. Moreover, the equation
\begin{equation}\label{eq1}
k^2-2am=2(k-a-1)
\end{equation} is hold.

On the other hand, we have
\begin{equation}\label{eqq}
k + (f_1-f_2)(\sqrt{2}) + (g_1-g_2)\sqrt{k^2-2am} = 0\end{equation}
 from  \cite[equation (2)]{HKM}. 

By equation (\ref{eqq}) we have $k = (g_2-g_1)\sqrt{k^2 - 2am}$. Since $m$ is an even integer, then $g_1+g_2 = m-1$ and $g_2-g_1$ are odd integers. 

Denote $g_2-g_1$ by $t$ and $\sqrt{k^2-2am}$ by $s$. Then $k = ts$ and $k^2 - 2am = s^2$. Let's put these expressions into equation (\ref{eq1}). 

Since $s^2 = 2(ts-a-1)$, then $s = t \pm \sqrt{t^2-2(a+1)}$. 

Denote $\sqrt{t^2-2(a+1)}$ by $r$. Then $2a = t^2 - r^2 - 2$. Since $s=\sqrt{k^2-2am}$ then $2am = k^2-s^2 = (t^2-1)s^2$. Moreover, since $2a = t^2 - r^2 -2$ and $s = t\pm r$, then 
\begin{equation}\label{eq2}
(t^2 - r^2 - 2)m = (t^2 - 1)(t\pm r)^2. 
\end{equation}

Since
\begin{equation}\label{eq3}
(t^2 - 1 - (t^2 - r^2 -2))(t\pm r)^2 = (r^2+1)(t\pm r)^2,
\end{equation}
then 
the right part of equation (\ref{eq3}) is divided by $t^2 - r^2 - 2$.

Since $t$ and $r$ are odd integers, then $t^2 - r^2 - 2 = 8h+6$ and has a prime  divisor $p = 4i+3$. But $r^2 + 1$ can't be divided by $p$. Then $t\pm r$ is divided by $p$. Thus, $t^2 - r^2$ is divided by $p$ and $p$ can't be divisor of $t^2 - r^2 - 2$. It is a contradiction.
$\square$\end{proof}
\medskip

Let's  study the case $n\geq 7$.
Let $G$ be a DDG with parameters $(v,k,k-2,a,m,n)$ and $n\geq 7$.  Let $B$ be a class of the canonical partition of $G$. Let
$X=\{x_1, x_2, x_3, x_4\}$ be a set from $B$. Denote $W(X)=\bigcap_{x\in X}N(x)$ by $W$. Since $b=k-2$, then $k-6\leq |W| \leq k-2$. Let's consider these cases one by one.

\begin{lem}\label{k-2}
$|W| \neq k-2$.
\end{lem}
\begin{proof} Let $|W| = k-2$. Then the intersection of any two sets $N(x_i)\setminus N(x_1)$ and $N(x_j)\setminus N(x_1)$ are empty, $i\neq j$, $i,j\in \{2,3,4\}$. 

Let $y\in B\setminus X$. If $y$ has less than $k-2$ neighbors in $W$, then $y$ has at least one neighbor in $N(x_i)\setminus N(x_1)$ for each $i\in \{2,3,4\}$.   So $y$ has $k-2$ neighbors in $N(x_1)$ and at least $3$ neighbors out of $N(x_1)$. It is a contradiction. 

Hence, any $y$ in $B\setminus X$ has  $k-2$ neighbors in $W$ and $W = W(B)$. It follows by lemma \ref{n_divides_W} that $|W(B)|=n$ divides $k-2$. On the other hand,  $n$ divides $k^2 - 2$ by lemma \ref{Basic}. Hence $n$ divides $2$. It is a contradiction, because $n\geq 7$.
$\square$\end{proof}

\begin{lem}\label{k-6}
$|W| \neq k-6$.
\end{lem}
\begin{proof} Let $|W| = k-6$. 
The intersection of any two sets $N(x_1)\setminus N(x_i)$ and $N(x_1)\setminus N(x_j)$ are empty, $i\neq j$, $i,j\in \{2,3,4\}$. 

Let $y\in B\setminus X$. If $y$ has more than $k-8$ neighbors in $W$ then $y$ has at least one non-adjacent vertex in each set $N(x_1)\setminus N(x_i)$.
It is a contradiction, because $y$ has exactly two non-adjacent
vertices in $N(x_1)$. Thus, any vertex $y$ in $B\setminus X$ has $k-8$ neighbors in $W$.
We have $n-4$ possibilities for choosing $y$ in $B\setminus X$.
Hence, $|W(B)| = |W| - 2(n - 4) = k + 2 - 2n$. By lemma \ref{n_divides_W} $n$ divides $k+2$. As in the previous lemma $n$ divides $k^2 - 2$. It is a contradiction.
$\square$\end{proof}

For any two vertices $x_i, x_j\in B$ we denote the set 
$(N(x_i)\cap N(x_j))\setminus W$ by $U_{i,j}$. 

\begin{lem}\label{k-5}
$|W| \neq k-5$.
\end{lem} 
\begin{proof} Let $|W| = k-5$.
We  have exactly three vertices in $U_{i,j}$ 
 for any $x_i, x_j\in X$, $i\neq j$.

Since $|W| = k-5$, then the remaining two vertices in $X\setminus \{x_i, x_j\}$
have no common neighbors in $U_{i,j}$. Thus, one of these two vertices has the only neighbor in $U_{i,j}$.

Let's consider $x_1, x_2$ in $X$, and let $x_3$ be a vertex which
has the only neighbor $u$ in $U_{1,2}$. Then $N(x_3)$ contains
$N(x_1)\setminus N(x_2)$ and $N(x_2)\setminus N(x_1)$. Therefore, $N(x_3)$ is contained in $N(x_1, x_2)$.

Since $u\notin W$, then $u\notin N(x_4)$ and 
$U_{1,2}\cup U_{1,3}\cup U_{2,3}=N(x_1, x_2)\setminus W$.  However, each pairs $x_4, x_1$, $x_4, x_2$, and $x_4, x_3$ must have three common neighbors in 
$N(x_1, x_2)\setminus W$. It is impossible.
$\square$\end{proof}

\begin{lem}\label{k-4}
If $|W| = n-4$, then $|W(B)| = k-4$.
\end{lem}
\begin{proof} Let $|W| = n-4$.
We  have exactly two vertices in $U_{i,j}$ 
 for any $x_i, x_j\in X$, $i\neq j$ and the remaining two vertices in $X\setminus \{x_i, x_j\}$ have no common neighbors in $U_{i,j}$. Thus,  there are two possibilities for the remaining two vertices in $X$.
\begin{enumerate}
 \item[(a)] There is a vertex in $X\setminus \{x_i, x_j\}$ without neighbor in $U_{i,j}$.
\item[(b)] Each vertex in $X\setminus \{x_i, x_j\}$ has the only neighbor in $U_{i,j}$ and $N(x_i)\cap U_{i,j}\neq N(x_j)\cap U_{i,j}$.
\end{enumerate}

 (a) If $x_k\in X\setminus \{x_i, x_j\}$ and $x_k$ has no neighbors in $U_{i,j}$,
 then vertices $x_i$, $x_j$ and $x_k$ have exactly $k-4$ common neighbors. Thus,  each vertex $y\in B\setminus \{x_i,x_j,x_k\}$ is  adjacent to all vertices in $W$, otherwise $|W(x_i,x_j,x_k,y)| \leq k-5$. It is impossible by the previous lemmas. Hence, $W(B) = k-4$. 

 (b) Now for any set of four vertices $X$ from $B$ and for any pair $\{x_i, x_j\}$ in $X$ each  vertex in $X\setminus \{x_i, x_j\}$ has  the only neighbor in $U_{i,j}$ and $N(x_i)\cap U_{i,j}\neq N(x_j)\cap U_{i,j}$.
 
 Let $B=\{x_1, x_2, \ldots , x_n \}$ and $X_s=\{x_1, x_2, x_3, x_4, \ldots , x_s\}$ for each $s\in 5, \ldots , n\}$. 
For any two different vertices $x_i, x_j\in X_s$ denote the union of all $U_{i,j}$ by $U_s$. Let $W_s = W(X_s)$ and $Z_s=N(X_s)\setminus (W(X_s)\cup U_s)$. It is significant that $N(X_s)$ is disjoint union of $W_s, U_s$ and $Z_s$. Moreover, each vertex $x_i$ from $X_s$ is adjacent to all vertices in $W_s$, and all vertices exclude one vertex (say $u_i$) in $U_s$, and exactly one vertex (say $z_i$) in $Z_s$. 
\smallskip

Let $|W(B)| \neq k-4$. We prove by induction on $s$ the following:

{\em For any $4< s \leq k$ there is a set $X_s$ of $s$ vertices in
$B=\{x_1, x_2, \ldots , x_n \}$ such that $|W_s| = k-s,\ |U_s| = s,\ |Z_s| = s$,
each vertex $x_i$ from $X_s$ is adjacent to all vertices in $W_s$, and all vertices exclude the only vertex (say $u_i$) in $U_s$, and one vertex (say $z_i$) in $Z_s$.}
\smallskip

We have $|X|=4$ and each vertex $x\in X$ has all $k-4$ neighbors from $W$, 3 neighbors in $U$, and $|Z| = 4$. By (b) each vertex from $U$ is adjacent to three vertices in $X$. Hence $\displaystyle 4\leq |U| \leq \binom{4}{3} = 4$. Thus, $|U| = 4$.
We take this statement as the basis step of induction.

 Let's append vertices one by one to $X$ and let $X_s$ be a subset of $B$  such that $|W_s| = k-s$, $|U_s| = s$ and $|Z_s| = s$. Moreover, $W_s\cup (U_s\setminus \{u_i\})\cup \{z_i\} = N(x_i)\cap (W_s\cup U_s\cup Z_s)$ for any $x_i\in X_s$. 
 
  Let's append $x_{s+1}$ to $X_s$.
If $|N(x_{s+1})\cap W_{s+1}| < k - (s + 1) $, then $x_1$, $x_2$, $x_3$ have $s-3$ common neighbors in $U_s$ and $x_1,x_2,x_3,x_{s+1}$ have  at most $k-s-2$ common neighbors in $W_s$. Hence, $|W(x_1,x_2,x_3,x_{s+1})| \leq k-5$ and we have a contradiction by Lemmas \ref{k-6}, \ref{k-5}. Thus, $k - (s + 1)\leq |N(x_{s+1})\cap W_{s+1}|$.
 
Let $|N(x_{s+1})\cap W_s| = k - s$. Since $x_i$ and $x_{s+1}$ for $i\leq s$ have $k-s$ common neighbors in $W_s$, then $ s-3 \leq |N(x_s)\cap U_s| \leq s-2$.
If $x_{s+1}$ adjacent to $u_s$, then $x_{s+1}$ and $x_s$ have at least  $s-3$ common neighbors in $\{u_1,\ldots, u_{s-1}\}$. Hence $|N(x_s)\cap U_s| = s-2$.  Let   $\{u_3,\ldots, u_s\}\subset N(x_{s+1})$. 

Consider the vertices $x_1$, $x_3$, $x_4$ and $x_{s+1}$. These four vertices have $k-s$ common neighbors in $W_s$ and $s-4$ common neighbors $\{u_5,\ldots,u_s\}$ in $U_s$. Hence $|W(x_1,x_3,x_4,x_{s+1})| = k-4$ but $x_3$, $x_4$ and $x_{s+1}$ also have $k-4$ common neighbors, because $x_{s+1}$ is non-adjacent to vertices in both sets $u_1,u_2$ and $\{u_1,u_2\} = U_{3,4}$. Then we have the same situation as in case (a). Hence, $W(B) = k-4$. It is a contradiction.

Let $|N(x_{s+1})\cap W_s| = k - s - 1$. Study the
neighbors of $x_{s+1}$ in $U_s$. If $x_{s+1}$ is non-adjacent to vertices $u_i,u_j\in U_s$, then  $x_{s+1}$ has at least $k-3$ common neighbors ($k-s-1$ in $W_s$, and $s-3$ in $U_s$, and maybe one in $Z_s$) with $x_s$, where $s\neq i,j$. If  $x_{s+1}$ is non-adjacent to the only vertex $u_i\in U_s$, then $x_{s+1}$ and $x_j$ has $k-3$ common neighbors in $W_s$, for any $i$, where $i\neq j$. So $x_{s+1}$ adjacent to all vertices from $Z_s\setminus\{z_i\}$. Hence $|N(x_{s+1})| = (k-s-1)+(s-1)+(s-1) = k+s-3>k$.
It is a contradiction.

Thus, $x_{s+1}$ adjacent to all vertices from $U_s$ and $x_{s+1}$ has one more neighbor outside $W_s\cup U_s$. Denote this neighbor by $z_{s+1}$. Also denote a vertex from $W_s\setminus N(x_{s+1})$ by $u_{s+1}$. Thus, the induction step is proved.

If we add all $n$ vertices from $B$ to $X$ then we have $|W(B)| = k-n$. But $n$ divides $|W(B)|$ by lemma \ref{n_divides_W} and $n$ divides $k^2 - 2$ by lemma \ref{Basic}. It is a contradiction. 
$\square$\end{proof}

\begin{lem}\label{k-3}
If $|W| = k-3$, then $|W(B)|$ equals either $k-3$ or $k-4$.
\end{lem}

\begin{proof} 
By Lemmas \ref{k-2} -- \ref{k-5} there are no fours vertices $Y$ in $B$ such that $|W(Y)|\in \{k-2, k-5, k-6\}$. If there is four vertices $Y$ in $B$ for which $|W(Y)|= k-4$, then $|W(B)| = k-4$ by Lemma \ref{k-4}. Thus, for any
four of vertices $Y$ in $B$ we have $|W(Y)|= k-3$. 

Let $|W| = n-3$. Let $u$ be the only vertex in $U_{1,2}$. If $u$ is adjacent to a vertex in $B\setminus \{x_1, x_2\}$, then we can find four vertices $Y$  in $B$ such that $|W(Y)| < k-3$. Hence,  $u\notin W(Y)$ for any fours vertices $Y$ in $B$ and $|W(B)| = k-3$. $\square$\end{proof}

\begin{lem}\label{|W(B)|} 
For $|W(B)|$ and $n$ one of the following cases holds:
\item{$(i)$} If $|W(B)| = k-3$ then $n=7$;
\item{$(ii)$} If $|W(B)| = k-4$ then $n=7$ or $n=14$.
\end{lem}
  \begin{proof} 
  By Lemma  \ref{n_divides_W} $n$ divides $|W(B)|$.  By Lemma \ref{Basic}(ii) $k^2 - 2$ is divided by $n$. Hence $n$  divides $7$ in the first case and $n$  divides $14$ in the second case. This two cases can't hold simultaneously, because $k-3$ and $k-4$ are mutually prime. 
$\square$\end{proof}

Denote $N(B)\setminus W(B)$ by $B'$. 

\begin{rem}\label{remark}
Since the canonical partition of a DDG is equitable (by lemma \ref{equitable}), if $y\not \in W(B)$, then $B \cap W(B_y) = \emptyset$. This fact implies that each vertex $y$ from $B'$ can't have more than $k - |W(B)|$ neighbors in $B$.
\end{rem}

\begin{lem}\label{class}
$B'$ is a class of the canonical partition of $G$ and $n=7$.
\end{lem}
  \begin{proof}    By Lemma \ref{|W(B)|} we have two cases.
  
  $(i)$ Let $|W(B)| = k-3$ and $n=7$. We have $|B|=7$ and $3$ edges from any vertex of $B$ to $B'$. Therefore, there are exactly $21$ edges between $B$ and $B'$.
  Moreover, each pair of different vertices of $B$ has one common neighbor in $B'$.
  There are $\displaystyle\cfrac{1}{2}\cdot \binom{7}{2} = 21$ pairs of different vertices of $B$.  Thus, there are $\displaystyle\cfrac{1}{3}\cdot \cfrac{1}{2}\cdot\binom{7}{2} = 7$ vertices in $B'$. 
 
  $(ii)$ Let $|W(B)| = k-4$ and either $n=7$ or $n = 14$. 
   If $x \in B$ and $Y = N(x)\cap B'$, then $x$ and any $y\in B\setminus \{x\}$ have two common neighbors in $B'$. Let's calculate the number of edges between $Y$ and $B$ in two ways. If $n = 14$, then we have $2(n-1) = 26$ edges between $B$ and $Y$ and at most  $12$ edges between $Y$ to $B$. It is a contradiction. 
  Let $n = 7$. We have $6\cdot 2 = 12$ edges from $B$ to $Y$ and at most $4\cdot 3 = 12$ edges from $Y$ to $B$. Then each vertex from $B'$ has $4$ neighbors in $B$.
  There are $28$ edges between $B$ and $B'$, hence $|B'|= 7$.
   $\square$\end{proof}
    \bigskip
  
   Now we can finish our proof of Proposition \ref{a not 0}.  There is a unique class  $B'$ for each class $B$ in $G$. Any two vertices $x,y$, where $x$ from $B$ and $y$ from $B'$, have no common neighbors in $B\cup B'$. Hence $N(x)\cap N(y) = W(B)\cap W(B')$ and by Lemma \ref{n_divides_W} $n$ divides $a$. 
  Let $|W(B)| = k-3 > 0$ and let  $z$ be a vertex from $W(B)$. Since $z$ adjacent to all vertices in $B$, then $z$ and $y$ has $3$ common neighbors  in $B$ and maybe there are $3$ common neighbors of $y$ and $z$ in $B_z'$. Then $a = |N(y)\cap N(z)|$ and $a \equiv 3$ or $6$ modulo $n$. It is a contradiction. Thus, there are only classes $B$ and $B'$ in $G$ and $a=0$. 
  
  If $|W(B)| = k-4 > 0$ then we have the same contradiction because $a = |N(y)\cap N(z)| \equiv 4$ or $1$ modulo $n$. Therefore, the case $a\neq 0$ is impossible and Proposition \ref{a not 0} is proved.
  $\square$\end{proof}
 
  Let's prove Theorem 3.  
  Let $G$ be a DDG with parameters 
  $(v,k,\lambda_1 ,\lambda_2 ,m,n)$. If $G$ is a disconnected DDG, then by \cite[Proposition 4.3]{EFHHH} each component of $G$ is a strongly regular graph with parameters $(v,k,\lambda)$. 
  There are no such graphs with $\lambda =k-2$.
 If  $G$ is a connected DDG with $0\in \{\lambda_1 ,\lambda_2\}$, then by Propositions \ref{a=0} $G$ is isomorphic to one of the following graphs: the incidence graph of the Fano plane, the non-incidence graph of the Fano plane $G$, and the complement of the three-dimensional binary cube.  By Proposition \ref{a not 0} there are no DDGs with  $k-2\in \{\lambda_1 ,\lambda_2\}$ and $0\notin \{\lambda_1 ,\lambda_2\}$.
 \bigskip
 
{\bf ACKNOWLEDGEMENTS}
\medskip

We are grateful to the Mikhail Lepchinskiy for useful remarks.
\medskip

{\bf ORCID}

Vladislav V. Kabanov http://orcid.org/0000-0001-7520-3302

Leonid Salaginov  https://orcid.org/0000-0001-6912-2493

\end{document}